\documentclass[10pt]{amsart}

\usepackage{verbatim}

\begin{document}



\newcommand{\dtime}{\frac{d}{d t}} 
\newcommand{\Lie}[2]{\mathcal{L}_{#1} #2} 
\newcommand{\dvar}[1]{\partial_{#1}} 
\newcommand{\tilcov}{\tilde{\nabla}} 
\newcommand{\lin}{\left.\frac{\partial}{\partial s} \right|_{s=0}} 


\newcommand{\G}[3]{\Gamma_{#1 #2}^{#3}} 
\newcommand{\tilG}[3]{\tilde{\Gamma}_{#1 #2}^{#3}} 
\newcommand{\amb}{\mathcal{O}} 
\newcommand{\sP}{\mathcal{P}} 

\newcommand{\sZ}{\mathcal{Z}}
\newcommand{\sT}{\mathcal{T}}
\newcommand{\sQ}{\mathcal{Q}}

\newcommand{\bR}{\mathbb{R}}


\newcommand{\tr}{\mathrm{tr}}       

\newenvironment{todo}{\begin{list}{$\Box$}{}}{\end{list}}


\newtheorem{main}{Theorem}
\renewcommand{\themain}{\Alph{main}}
\newtheorem{thm}{Theorem}[section]
\newtheorem{remark}{Remark}[section]
\newtheorem{lemma}[thm]{Lemma}
\newtheorem{prop}[thm]{Proposition}






\title[Geometric flows]{Short-Time Existence for Some Higher-Order Geometric Flows}
\author{Eric Bahuaud}
\address{Department of Mathematics,
Stanford University.}
\email{bahuaud(at)math.stanford.edu}

\author{Dylan Helliwell}
\address{Department of Mathematics, 
Seattle University.}
\email{helliwed(at)seattleu.edu}

\date{\today}

\begin{abstract}
We establish short-time existence and regularity for higher-order flows generated by a class of polynomial natural tensors that, after an adjustment by the Lie derivative of the metric with respect to a suitable vector field, have strongly parabolic linearizations.  We apply this theorem to flows by powers of the Laplacian of the Ricci tensor, and to flows generated by the ambient obstruction tensor.  As a special case, we prove short-time existence for a type of Bach flow.
\end{abstract}

\maketitle


\section{Introduction}

Geometric Flows have played an important role in geometry for many years.  Perhaps the most famous such flow is Ricci flow, introduced by Richard Hamilton.  In this article, we prove a general existence theorem for high-order geometric flows on compact manifolds.  We study the flow
\begin{equation*}
\left\{
\begin{aligned}
\dvar{t} g &= T \\
g(0) &= h
\end{aligned}
\right.
\end{equation*}
where $T$ is a polynomial natural symmetric two-tensor and $h$ is a smooth initial metric.  Motivated by Ricci Flow, we focus on those flows that, after adjusting by the Lie derivative of $g$ with respect to a well chosen vector field, have strongly parabolic linearizations, allowing us to make use of the DeTurck trick.  For such flows, we have the following short-time existence and regularity result:

\begin{main} \label{maintheorem}
Let $T$ be a polynomial natural symmetric two-tensor for a single metric $g$ in dimension $n$ with the property that there is a polynomial natural vector field $W$ for two metrics $g$ and $h$ in dimension $n$, such that the linearization of $T+ \Lie{W}{g}$ at any metric $h$ is a strongly elliptic linear operator of order $2m$.  Then, on any compact manifold of dimension $n$ and any initial metric $h$, there is a smooth short-time solution to
\begin{equation}
\tag{$\ast$}\label{originalflow}
\left\{ \begin{aligned} 
\dvar{t} g &= T \\
g(0) &= h.
\end{aligned} \right.
\end{equation}
\end{main}

There are two main families of flows to which we apply Theorem \ref{maintheorem}.  The first family generalizes Ricci flow to allow for powers of the Laplacian.  For this, we note that our convention for powers of the Laplacian is
\[
\Delta^p = g^{j_1 k_1} \cdots g^{j_p k_p} \nabla_{j_1} \nabla_{k_1} \cdots \nabla_{j_p} \nabla_{k_p}.
\]

\begin{main} \label{plaprictheorem}
Let $h$ be a smooth metric on a compact manifold M and let $p \geq 0$ be an integer.  Then there is a smooth short-time solution to the following flow:
\begin{equation}
\tag{$\ast \ast$} \label{plapricflow}
\left\{
\begin{aligned}
\dvar{t} g &= 2 (-1)^{p+1} \Delta^p Ric\\
g(0) &= h.
\end{aligned}
\right.
\end{equation}
\end{main}

Note that when $p = 0$, we have Ricci flow, motivating the choice of the 2 in the flow equation.  When $p=1$, the steady state solutions are ``Ricci-harmonic'' manifolds.

The second family involves the ambient obstruction tensor $\amb_n$ introduced by Charles Fefferman and Robin Graham.  Note here that $S$ is the scalar curvature of $g$.  

\begin{main} \label{ambienttheorem}
Let $h$ be a smooth metric on a compact manifold of even dimension $n \geq 4$ and let
\[
c_n = \frac{1}{2^{\frac{n}{2}-1}\left(\frac{n}{2}-2\right)!(n-2)(n-1)}.
\]
Then there is a smooth short-time solution to the following flow:
\begin{equation}
\tag{$\ast \! \ast \! \ast$} \label{ambflow}
\left\{
\begin{aligned}
\dvar{t} g &= \amb_n + c_n (-1)^{\frac{n}{2}} (\Delta^{\frac{n}{2}-1} S) g  \\
g(0) &= h.
\end{aligned}
\right.
\end{equation}
\end{main}

Note that when $n = 4$, the ambient obstruction tensor is the Bach tensor, and so we have a type of Bach flow.  While this flow may not be the first choice for ``Ambient Obstruction Flow'' (because of the additional $\Delta^{\frac{n}{2}-1} S g$ term) it does have the nice property that because the ambient obstruction tensor is trace free, steady state solutions cause the ambient obstruction tensor and the given power of the Laplacian of $S$ to vanish, which implies that on a compact manifold, the scalar curvature must be constant.

This paper is organized as follows.  In Section \ref{background}, we introduce our notation and develop the background for studying the geometric flows in which we are interested.  Ricci flow, which motivated this work is discussed and the so called DeTurck trick is generalized.  Polynomial natural tensors are defined and characterized.  Finally, the ambient obstruction tensor is introduced and its properties are discussed.

The necessary parabolic Schauder theory is discussed in Section \ref{schauder}.  H\"{o}lder spaces for the parabolic setting in question are defined.  Existence results for higher-order linear systems are stated.  Lastly, the regularity result necessary to our setting is provided.

In Section \ref{mainthm}, Theorem \ref{maintheorem} is proved.  This is achieved by conditioning the equation, extracting the linear part, using parabolic Schauder theory for high-order systems, and then applying the contraction mapping theorem.

Finally, in Section \ref{applications}, we prove Theorems \ref{plaprictheorem} and \ref{ambienttheorem}.  In each case, this is done by finding a vector field that removes certain undesirable terms so that the linearization of the adjusted flows are strongly parabolic.  Theorem \ref{maintheorem} then applies.  Once Theorem \ref{ambienttheorem} is proved, we discuss the effect of a conformal transformation, motivated by the fact that the ambient obstruction tensor is conformally invariant.

The authors would like to thank Robin Graham, Rafe Mazzeo and Andr\'as Vasy for the many helpful discussions during the course of this work.  The second author is also grateful to the members of the Mathematics department at Stanford University for their hospitality during his sabbatical visit.


\section{Background} \label{background}

Throughout this paper, $M$ will be a smooth compact manifold of dimension $n$ admitting a smooth Riemannian metric $h$.  This metric will serve as our background metric and our initial metric for the various flow problems we consider.  Quantities associated with $h$ will be indicated with a tilde.  The metric $g$ will be used to denote a solution to the flow problems or for generic arguments.  Quantities associated with $g$ will be unadorned.  We will denote the bundle of symmetric two tensors on $M$ by $\Sigma^2(M)$.  Our convention for the Laplacian is given for example by
\[ \widetilde{\Delta} u = h^{ij} ( \partial^2_{ij} u - \tilG{i}{j}{k} \partial_k u). \]
It is to be understood that when working with various flow equations, the quantities being used are for a specific time slice (as opposed to a quantity on the manifold $M \times [0,T]$).

\subsection{Ricci Flow and the DeTurck Trick}

Introduced by Richard Hamilton, Ricci Flow has enjoyed a great deal of popularity, owing to its central role in Grigory Perelman's proof of William Thurston's geometrization conjecture.  See \cite{ChowLN, MorganTian, Topping} and the references therein for more about this important flow.  

Ricci flow is the following:
\begin{equation*}
\left\{
\begin{aligned}
\dvar{t} g &= -2 Ric \\
g(0) &= h.
\end{aligned}
\right.
\end{equation*}

This flow is only weakly parabolic because of the diffeomorphism invariance of $Ric$.  Hamilton overcame this problem by making use of the Nash Moser Inverse Function Theorem \cite{Hamilton}.  Later, Dennis DeTurck developed a technique to circumvent this problem and establish short-time existence using simpler methods \cite{DeTurck}.  The so called DeTurck trick involves finding an appropriate time-dependent vector field $W$ such that adding the Lie derivative of the metric with respect to $W$ creates a strongly elliptic operator.

Let $A$ be the difference tensor:
\[
A_{ij}^k = \Gamma_{ij}^k - \tilde{\Gamma}_{ij}^k.
\]
Define a time-dependent vector field $W$:
\[
W^k = g^{pq} (\Gamma_{pq}^k - \tilde{\Gamma}_{pq}^k),
\]
and consider the flow
\begin{equation*}
\left\{
\begin{aligned}
\dvar{t} g &= -2 Ric + \Lie{W}{g} \label{ricdeturck} \\
g(0) &= h.
\end{aligned}
\right.
\end{equation*}

This adjusted flow can be solved by standard techniques.  To produce a solution to the original flow, the solution to the adjusted flow is pulled back by diffeomorphisms generated by $-W$.  This process works in a more general context and we will need it for our results.  The necessary fact is the following:

\begin{lemma} (The DeTurck Trick) \label{deturck}
Given a smooth natural tensor $T$ for a single metric $g$ (see below for the definition of a natural tensor), a fixed metric $h$, and a smooth time-dependent vector field $W$, let $g$ be a smooth one-parameter family of metrics that solves
\begin{equation*}
\left\{
\begin{aligned}
\dvar{t} g &= T(g) + \Lie{W}{g} \\
g_0 &= h.
\end{aligned}
\right.
\end{equation*}
Let $\theta$ be the flow generated by $-W$.  Then
$\bar{g} = \theta_{(t,0)}^*(g)$ solves
\begin{equation*}
\left\{
\begin{aligned}
\dvar{t} \bar{g} &= T(\bar{g}) \\
\bar{g}_0 &= h.
\end{aligned}
\right.
\end{equation*}
\end{lemma}

\begin{proof}
In what follows, we append subscripts to indicate which time slice is being considered.  Note that the initial condition is satisfied since $\theta_{(0,0)} = id$.  Now, at a point $p \in M$ and a time $s$, we have
\begin{align*}
\left. \dtime \right|_{t=s} ({\bar{g}_t})_p &= \left. \dtime \right|_{t=s} (\theta_{(t,0)}^*(g_t))_p \\
		&= \left( \theta_{(s,0)}^*\left( \left. \dtime \right|_{t=s} g_t + \Lie{-W_s}{g_s} \right) \right)_p \\
		&= \left( \theta_{(s,0)}^*(T(g_s)+ \Lie{W_s}{g_s}) \right)_p
			- \left( \theta_{(s,0)}^* \Lie{W_s}{g_s} \right)_p \\
		&= \left( T(\theta_{(s,0)}^* g_s) \right)_p \\
		&= (T(\bar{g}_s))_p.
\end{align*}
(See \cite{Jack2} for details surrounding manipulations of time-dependent vector fields.)
\end{proof}

It should be noted that this is really just the first half of the DeTurck Trick.  The second half involves using harmonic map heat flow to prove uniqueness.  We do not generalize or try to make use of the second half of the DeTurck Trick.

\subsection{Polynomial Natural Tensors}

We say a tensor $T = T(g_1, \ldots, g_m)$ is a \textit{natural tensor} for $m$ metrics in dimension $n$ if for any $n$-manifold $M$ and any diffeomorphism $f$ of $M$, we have $T(f^*(g_1), \ldots, f^*(g_m)) = f^*(T(g_1, \ldots, g_m))$ (i.e. it is preserved by isometries).  We say a natural tensor is \textit{polynomial} if, in any coordinate system, its coefficients are polynomials in the coefficients of the coordinate derivatives of $g_i$ up to some $\partial^k g_i$, and $g_i^{-1}$.  When necessary, we will say a polynomial natural tensor is \textit{of order k} to indicate the highest order derivatives involved in the expression and we will indicate the order with a subscript.  As examples, $Ric$ is a polynomial natural tensor of order two for one metric in any dimension, the difference tensor $A$ is a polynomial natural tensor of order one for two metrics in any dimension, and the Bach tensor is a polynomial natural tensor of order four for a single metric in dimension 4.

In order to make use of existence results for parabolic systems, we will need to express all the derivatives in a given polynomial natural tensor $T$ as covariant derivatives with respect to the background metric $h$.  The necessary result is the following.

\begin{prop} \label{invariantsWRTh}
Let $T(g_1, g_2, \ldots, g_m)$ be a polynomial natural tensor and let $h$ be a background metric.  Then $T$ can be expressed as a linear combination of contractions (and sharpenings and flattenings) with respect to the $g_i$ of covariant derivatives with respect to $h$ of the $g_i$ and curvature for $h$.
\end{prop}

One way to prove this proposition is to work at the origin in normal coordinates for $h$.  Then all the coordinate derivatives in $T$ can be written in terms of covariant derivatives with respect to $h$.  In the process, covariant derivatives of the curvature of $h$ are also introduced.  See \cite{AtiyahBottPatodi} or \cite{Epstein} for more details surrounding natural tensors.

For our application, the background metric will coincide with one of the metrics defining the polynomial natural tensor in question.  Also, we note that if $W$ is a polynomial natural vector field, then
$\Lie{W}{g}$ is also a polynomial natural tensor.

Finally, while working with tensors in coordinates and to represent terms of order less than the principal part that appear in various equations, we will occasionally have a need to encapsulate terms depending on a certain number of derivatives of other tensors.  With this in mind, let $\sP(\partial^{p_1} s_1, \partial^{p_2} s_2,\ldots)$ denote an expression whose components are polynomials in the components of the tensors $s_r$ and their coordinate derivatives up to order $p_r$.

\subsection{The Ambient Obstruction Tensor}

For each even $n \geq 4$ the \textit{ambient obstruction tensor} of order $n$ is given by
\begin{equation} \label{obstructiondef}
\amb_n = \frac{1}{(-2)^{\frac{n}{2}-2} \left(\frac{n}{2} - 2\right)!}
		\left(\Delta^{\frac{n}{2}-1} P - \frac{1}{2(n-1)} \Delta^{\frac{n}{2}-2} \nabla^2 S \right) + T_{n-1}
\end{equation}
where
\begin{equation*}
P = \frac{1}{n-2} \left( Ric - \frac{1}{2(k-1)} S g \right)
\end{equation*}
is the Schouten tensor and $T_{n-1}$ is a polynomial natural tensor of order $n-1$, the specific structure of which we do not need.

The ambient obstruction tensor was introduced by Charles Fefferman and Robin Graham as the obstruction to a formal expansion of an asymptotically hyperbolic Einstein metric with a given conformal infinity in dimension $n+1$ \cite{FeffG}.  When $n = 4$, the ambient obstruction tensor is the Bach tensor.  In general, it is a polynomial natural tensor of order $n$ for a single metric in dimension $n$.  Moreover, it is symmetric, trace free, and divergence free.  Finally, it is conformally invariant with weight $2-n$, where we say a tensor $T$ is conformally invariant of weight $w$ if for $\hat{g} = \rho^2 g$, we have $\hat{T} = \rho^w T$, where $\rho$ is a smooth positive function.  See \cite{GHir} for more details.


\section{Parabolic Schauder theory} \label{schauder}

In this section we state the existence and regularity results for linear parabolic equations of order $2m$ on compact manifolds that we will need.  We review recent work of Tobias Lamm \cite{Lamm} for high-order parabolic equations that adapts to our setting.  

We first restrict to $\bR^n \times \bR_{+}$.  Define the parabolic distance between $(x_1, t_1)$ and $(x_2, t_2)$ by
\[ d( (x_1, t_1), (x_2, t_2)) = \max \{ |x_1-x_2|, |t_1-t_2|^{\frac{1}{2m}} \}. \]
For a smooth bounded open set $U  \subset \bR^n \times \bR_{+}$ and a function $u: \bR^n \times \bR_{+} \longrightarrow \bR^N$ define a seminorm:
\[ [u]_{\alpha, U} := \sup \left\{ \left. \frac{|u(p)-u(q)|}{d(p,q)^{\alpha}} \right| p, q \in U \; \mbox{and} \; p \neq q \right\}. \]
Now, for a non-negative integer $a$ and $b \in \{0, 1\}$ we define
\begin{align*}
C^{a,b}(U) &:= \left\{ u: U \rightarrow \bR^N | u, \nabla u, \cdots, \nabla^{a} u, \cdots, \partial_t^b u \in C^0(U)\right\}, \\
[D^{a,b} u]_{\alpha, U} &:= [\nabla^{a} u]_{\alpha, U} + [\partial^b_t u]_{\alpha, U}, \\
||u||_{C^{a,b;\alpha}(U)} &:= \sum_{i=0}^{a} ||\nabla^i u||_{L^{\infty}(U)} + \sum_{i=0}^{b}||\partial^i_t u||_{L^{\infty}(U)} + [D^{a,b} u]_{\alpha, U}, \\
C^{a,b;\alpha}(U) &:= \left\{ u \in C^{a,b}(U): ||u||_{C^{a,b;\alpha}(U)} < \infty \right\},
\end{align*}
where $\nabla$ may be interpreted as the covariant derivative with respect to the flat metric on $U$.  We note that the spaces $C^{0,0;\alpha}(U)$ and $C^{2m,1;\alpha}(U)$ are the basic spaces on which to prove parabolic Schauder estimates.  Our contraction mapping argument will occur on $C^{2m,0;\alpha}(U)$.  We also need purely spatial H\"older spaces, but we will not define these separately.

Consider a differential operator $L$ given coordinates $\{ x^1, \cdots, x^n, t \}$ by
\[ L = \sum_{|\beta| \leq 2m} a_{\beta} D^{\beta}, \]
where $\beta = (\beta_1, \cdots, \beta_n)$, $D^{\beta} = \partial^{\beta_1}_{x^1} \cdots \partial^{\beta_n}_{x^n}$, and $a_{\beta} = (a_{\beta}(x,t))^j_k$ for $1 \leq j, k \leq N$, where we emphasize that we are only allowing spatial derivatives.  We say $L$ is strongly elliptic if there is a constant $\Lambda > 0$ such that
\begin{equation} \label{strongel}
(-1)^m \sum_{|\beta| = 2m} (a_{\beta}(x,t))^j_k \xi^{\beta} \eta_j \eta^k \geq \Lambda |\xi|^{2m} |\eta|^2,
 \end{equation}
for all $(x,t) \in U$, $\xi \in \bR^n,$ and $\eta \in \bR^N$, and where the $(-1)^m$ prefactor comes from the fact that $D^{\beta}$ is mapped to $(i \xi)^{\beta}$ under the Fourier transform.  Note also the norms are with respect to the flat metric in coordinates.  We require $a^{\beta} \in C^{0,0;\alpha}$; in much of our application these coefficients are smooth and time-independent.  We say an operator $P = \partial_t + L$ is strongly parabolic if $L$ is strongly elliptic.

We now define H\"older spaces on bundles over compact manifolds modeled on our definition in Euclidean space.  Let $r$ be the injectivity radius of $M$.  Since $M$ is compact, $r > 0$.  Let $\{U_i \}$ be any (finite) covering by normal coordinate balls of radius less than $r$ with charts $\phi_i: U_i \rightarrow M$.  Let $\{\rho_i\}$ be a partition of unity subordinate to $\{ \phi_i( U_i ) \}$.  Given a section $u$ of $\Sigma^k(M) \times [0,T]$, we define a norm
\[ ||u||_{C^{2m,1;\alpha}(M \times [0,T])} := \sum_i || \phi_i^* (\rho_i u) ||_{C^{2m,1;\alpha}(U_i \times [0,T])}, \]
where the norm on the right hand side is computed with respect to the Euclidean metric, as above.  Let $C^{2m,1;\alpha}(M \times [0,T])$ be the completion of smooth sections with respect to this norm.  Note also that while $u$ is a section of a tensor bundle we will not explicitly mention this in the notation.  In the sequel we will call a choice of covering $(U_i, \phi_i, \rho_i)$ a reference cover.

It follows from the compactness of $M$ that the topology of $C^{2m,1;\alpha}(M \times [0,T])$ is well-defined as any two choices of reference cover for $M$ lead to equivalent norms.  We also document the mapping properties of the covariant derivative.  If $u$ is a symmetric $j$-tensor, then in any choice of coordinates, $\nabla^k u$ may be expressed as a differential operator of order $k$ acting on the components of $u$ with coefficients that depend on $k$-derivatives of the metric appearing through Christoffel symbols.  As the metric is smooth, any finite number of derivatives of the metric are uniformly bounded over charts in a reference cover.  This leads to the following
\begin{lemma} For any $1 < j \leq 2m$, we have bounded mappings
\[ \tilcov^j: C^{2m,1;\alpha}( M \times [0,T]) \longrightarrow C^{0,0;\alpha}( M \times [0,T]). \]
\end{lemma}
\begin{remark}
We caution the reader again that the notation we are using for the spaces does not keep track of the tensor order.
\end{remark}

A differential operator $L$ on a vector bundle over $M$ is strongly elliptic if given any choice of reference coordinates $(U_i, \phi_i)$, the operator $L_i = \phi_i^* \circ L \circ (\phi_i^{-1})^*$ is strongly elliptic.  We say $L$ has $C^{0,0;\alpha}$ coefficients if the coefficients of $L_i$ lie in $C^{0,0;\alpha}(U_i \times [0,T])$.

The basic a priori interior estimate for parabolic systems in Euclidean space is given in Lamm's work \cite[Satz 2.3.10, Satz 2.3.13]{Lamm}.  We state an existence result here for symmetric 2-tensors on compact manifolds without boundary, the proof is similar to the one given in \cite{Lamm}.

\begin{prop} \label{thm:linearparabolic} Let $L: C^{\infty}(\Sigma^2(M)) \rightarrow C^{\infty}(\Sigma^2(M))$ be a strongly elliptic differential operator with $C^{0,0;a}$ coefficients on a smooth compact manifold.  Then for every $f \in C^{0,0;\alpha}(M \times [0,T])$ there exists a unique solution $u \in C^{2m,1;\alpha}(M \times [0,T])$ of
\begin{equation*}
\left\{
\begin{aligned}
(\partial_t + L) u(x,t) &= f(x,t) \\
u(x,0) &= 0,
\end{aligned}
\right.
\end{equation*}
and we have the estimate
\[ ||u||_{C^{2m,1;\alpha}(M \times [0,T])} \leq K ||f||_{C^{0,0;\alpha}(M \times [0,T])}. \]
\end{prop}

We remark that the Schauder constant $K$ remains bounded as $T \rightarrow 0$.

Finally, we state a version of parabolic regularity.  See Friedman \cite{Friedman} or Krylov \cite{Krylov} for how to translate interior estimates (for second order parabolic operators) to regularity results.

\begin{prop} \label{thm:parabolicregularity} Let $L: C^{\infty}(\Sigma^2(M)) \rightarrow C^{\infty}(\Sigma^2(M))$ be a strongly elliptic differential operator such in any references coordinates $||D^{\gamma} a_{\beta}||_{C^{0,0;\alpha}(M \times [0,T])} \leq C$, for any multi-index $\gamma$ with $|\gamma| \leq k$.  Suppose also $D^{\gamma} f \in C^{0,0;\alpha}(M \times [0,T])$ for all $|\gamma| \leq k$ and that $u \in C^{2m,1;\alpha}(M \times [0,T])$ is a solution of $(\partial_t + L) u(x,t) = f(x,t)$ on $M \times [0,T]$.  Then $D^{\gamma} u \in C^{2m,1;\alpha}(M \times [0,T])$, for all $|\gamma| \leq k$.
\end{prop}


\section{The Main Theorem} \label{mainthm}

In this section we restate and prove our main theorem.  We are interested in studying the flow
\begin{equation*} 
\left\{ \begin{aligned} 
\dvar{t} g &= T \\
g(0) &= h
\end{aligned} \right.
\end{equation*}
where $T$ is a polynomial natural tensor of order $2m$.  Like the Ricci flow, this system may not be parabolic due to the diffeomorphism invariance of $T$.  We therefore consider a generalized Ricci-DeTurck trick.  We instead study
\begin{equation*} 
\left\{ \begin{aligned} 
\dvar{t} g &= T + \Lie{W}{g}\\
g(0) &= h
\end{aligned} \right.
\end{equation*}
where $W$ is a polynomial natural vector field for the two metrics $g$ and $h$, chosen so that the linearization of $T(g) + \Lie{W}{g}$ at $h$ is strongly elliptic.   

Our theorem is as follows:

\medskip
\noindent \textbf{Theorem \ref{maintheorem}.} \emph{Let $T$ be a polynomial natural symmetric two-tensor for a single metric $g$ in dimension $n$ with the property that there is a polynomial natural vector field $W$ for two metrics $g$ and $h$ in dimension $n$, such that the linearization of $T+ \Lie{W}{g}$ at any metric $h$ is a strongly elliptic linear operator of order $2m$.  Then, on any compact manifold of dimension $n$ and any initial metric $h$, there is a smooth short-time solution to
\begin{equation}
\tag{$\ast$}
\left\{ \begin{aligned} 
\dvar{t} g &= T \\
g(0) &= h.
\end{aligned} \right.
\end{equation}
}

\medskip

The adjustment by the Lie derivative allows us to use the DeTurck trick and the form of the tensor allows us to apply Schauder theory and a fixed point argument to prove existence.  Applying this theorem then requires finding an appropriate vector field with which to take the Lie derivative, and verifying that the adjusted tensor is of the right form.

The proof of Theorem \ref{maintheorem} will proceed in three main parts.  First we will show that the adjusted equation
\begin{equation} \label{adjustedflow}
\left\{ \begin{aligned} 
\dvar{t} g &= T + \Lie{W}{g} \\
g(0) &= h
\end{aligned} \right.
\end{equation}
can be conditioned so that all derivatives are expressed as covariant derivatives with respect to $h$.  This will introduce an inhomogeneous part.  Second, we will use the conditioned form of the adjusted equation to show that it has a solution for a short time in an anisotropic H\"older space with finite regularity.  After a bootstrap argument to obtain full smoothness, we will apply the DeTurck trick to prove that the original system has a solution.  The second part itself requires a number of steps.  We will find a solution to the linear part using Schauder theory for higher-order systems.  Then we will show, still using the Schauder estimates that the quadratic part can be made ``small'' by choosing a small enough neighborhood in the appropriate function space.  Finally we will deal with the inhomogeneous part by restricting to a short time interval.  Applying a contraction mapping argument, we will find a fixed point of the solution operator defined by the linear part.  This fixed point will be the desired solution.


\subsection{Conditioning the Adjusted Flow}

The first step to conditioning the adjusted flow is to invoke Proposition \ref{invariantsWRTh}.  This says we can write $T+ \Lie{W}{g}$ in terms of covariant derivatives of $g$ with respect to $h$, $g^{-1}$ and $h$ and its derivatives.

We now condition the equation further.  We will search for a solution to \eqref{adjustedflow} as a perturbation of the initial metric, i.e. in the form $g = h + v$.

We introduce the notation $\sT(g) := T(g) + \Lie{W}{g}$.  We form a Taylor expansion in $s$ for the expression $\sT(h + sv)$ about $s = 0$.  Define
\begin{align*}
L_h v &:= \frac{d}{ds} \left( \sT(h+sv)  \right) \big|_{s = 0}, \; \mbox{and} \\
Q( v, s) &:= \frac{1}{2} \frac{d^2}{d^2s} \left( \sT(h+sv) \right), \; \mbox{and}\\
\sQ(v) & := \int_0^1 Q( v, z ) dz.
\end{align*}
Now consider 
\begin{align*} 
\sT(h+sv) &= I_h + L_h v \cdot s + \int_0^1 Q( v, zs ) dz \cdot s^2,
\end{align*} 
where we have used the integral form of the remainder for the second order terms, and $I_h$ denotes the `inhomogeneous' term $I_h := \sT(h)$.  

Looking for a solution to \eqref{adjustedflow} of the form $g(x,t) = h(x) + v(x,t)$ then amounts to solving

\begin{equation} \label{eqn:eqnforv-1}
\left\{ \begin{aligned} 
\dvar{t} v &= \sT(h + v) = I_h + L_h v + \sQ(v), \\
v(0) &= 0.
\end{aligned} \right.
\end{equation}

We pause to remark at this point that for our eventual geometric applications, that is flow by powers of the Laplacian of the Ricci tensor and flow by the ambient obstruction tensors, it is possible to work out the expansion above quite explicitly once one has an expansion for the inverse of the evolving metric.  Indeed, our work was motivated by these expansions.

We next need to understand the mapping properties of the quadratic remainder term more closely.  We work in the space $C^{2m,0;\alpha}$.  Note that in our eventual application we will need this lemma for $v$ with very small norm.

\begin{lemma} There is a positive constant $C$ depending on $\sQ$ such that for $u, v \in C^{2m,0;\alpha}(M \times [0,T])$, with 
\[ ||u||_{C^{2m,0;\alpha}(M \times [0,T])}, ||v||_{C^{2m,0;\alpha}(M \times [0,T])} < 1, \]
we have the following
\begin{enumerate}
\item $\sQ$ satisfies the following quadratic estimate:
\begin{equation*} 
|| \sQ(v) ||_{C^{0,0;\alpha}(M \times [0,T])} \leq C ||v||^2_{C^{2m,0;\alpha}(M \times [0,T])}. 
\end{equation*}
\item $Q$ satisfies the following `difference of squares' Lipschitz-type estimate:
 \begin{align*} 
|| \sQ(u) &- \sQ(v) ||_{C^{0,0;\alpha}(M\times[0,T])} \nonumber \\
&\leq C \max\left\{ ||u||_{C^{2m,0;\alpha}(M \times [0,T])}, ||v||_{C^{2m,0;\alpha}(M \times [0,T])} \right\} || u - v ||_{C^{2m,0;\alpha}(M \times [0,T])}. 
\end{align*}
\end{enumerate}
Moreover, $C$ remains bounded as $T \rightarrow 0$.
\end{lemma}
\begin{proof}
We first observe that $\sT$ contains at most the $2m$-th $h$-covariant derivative of $v$, which is purely a spatial derivative.  By the special form of the expansion above, every term of $Q(v,s)$ is quadratic or higher in $v$ and its $h$-covariant derivatives up to order $2m$.  We begin by restricting to the quadratic case
\[ c(h, h^{-1}, s) \tilcov^j v \tilcov^k v, \]
for $0 \leq j, k < 2m $, and the dependence on $s$ is polynomial. 

A term of this form may be estimated by
\begin{align*}
|| c(h, &h^{-1}, s) \tilcov^j v \tilcov^k v ||_{C^{0,0;\alpha}(M \times [0,T])} \\
&\leq C(s) ||\tilcov^j v||_{C^{0,0;\alpha}(M \times [0,T])} || \tilcov^k v ||_{C^{0,0;\alpha}(M \times [0,T])} \\
& \leq C(s) ||v||_{C^{2m,0;\alpha}(M \times [0,T])}^2,
\end{align*}
where the constant $C$ depends on $h$ and is polynomial in $s$.  For $\sQ(v)$, we now estimate through the integral and find
\begin{align*}
|| \sQ(v)  ||_{C^{0,0;\alpha}(M \times [0,T])}
& \leq \int_0^1 || Q(v, s) ||_{C^{0,0;\alpha}(M \times [0,T])} ds \\
& \leq C'  ||v||_{C^{2m,0;\alpha}(M \times [0,T])}^2,
\end{align*}
where we have used the fact that the integral in $s$ will be bounded as the dependence in $s$ is polynomial and $C'$ now depends on the algebraic structure, for example the number of terms, occurring in the expression for $\sQ$.  

The terms with higher powers of $v$ may be estimated in a similar manner.  This proves the first estimate. 

The second estimate follows from the fact that polynomials are locally Lipschitz.  To describe this a little further, returning to the special structure of the terms we consider the difference
\[ c(h, h^{-1}, s) \tilcov^j u \tilcov^k u - c(h, h^{-1}, s) \tilcov^j v \tilcov^k v  \]
for $0 \leq j, k < 2m $.  A term of this form may be estimated by by interpolating appropriate terms so as to form the difference $u - v$.  For example, we may abstractly write
\begin{align*}
c(h, h^{-1}, s) \tilcov^j u \tilcov^k u &- c(h, h^{-1}, s) \tilcov^j v \tilcov^k v \\
&= c(h, h^{-1}, s) \tilcov^j u \tilcov^k u - c(h, h^{-1}, s) \tilcov^j u \tilcov^k v \\
&+ c(h, h^{-1}, s) \tilcov^j u \tilcov^k v - c(h, h^{-1}, s) \tilcov^j v \tilcov^k v \\
&= c(h, h^{-1}, s) \tilcov^j u (\tilcov^k u - \tilcov^k v) + c(h, h^{-1}, s) (\tilcov^j u - \tilcov^j v) \tilcov^k v
\end{align*}

We now estimate as before.  The fact that the constant $C$ remains bounded as $T \rightarrow 0$ follows from the fact that the only time dependence in $\sQ$ occurs though its arguments.
\end{proof}


\subsection{Solving the Conditioned Flow}

We now set up and the short-time existence for a certain class of flows.  Our basic equation for a symmetric two tensor may be written

\begin{equation}  \label{eqnforv}
\left\{ \begin{aligned}
\partial_t v + L v &= F(v, \tilcov v, \cdots,  \tilcov^{2m} v) + I, \\
		v(0) &= 0. 
\end{aligned} \right.
\end{equation}

where $L$ is strongly elliptic operator of order $2m$ with smooth coefficients that are independent of $t$, $I: \Sigma^2(M) \longrightarrow \Sigma^2(M)$ is a smooth map independent of time that acts as an inhomogeneous term, and $F(v) \;( = F(v, \tilcov v, \cdots,  \tilcov^{2m} v) )$ is a smooth function that satisfies the following properties: there exists a constant $C_F > 0$ such that
\begin{enumerate}
\item $F$ satisfies the following quadratic estimate:
\[ || F(v, \tilcov v, \cdots, \tilcov^{2m} v)||_{C^{0,0;\alpha}(M \times [0,T])} \leq C_F ||v||^2_{C^{2m,0;\alpha}(M \times [0,T])} \] 
\item $F$ satisfies the following `difference of squares' Lipschitz type-estimate:
\begin{align*} || F(&u, \tilcov u, \cdots, \tilcov^{2m} u)-F(v, \tilcov v, \cdots, \tilcov^{2m} v) ||_{C^{0,0;\alpha}(M \times [0,T])} \\ 
&\leq C_F \max\left\{ ||u||_{C^{2m,0;\alpha}(M \times [0,T])}, ||v||_{C^{2m,0;\alpha}(M \times [0,T])} \right\}\! ||u - v||_{C^{2m,0;\alpha}(M \times [0,T])}. \end{align*}
\end{enumerate}

We now explain the contraction mapping argument that leads to short-time existence for \eqref{eqnforv}.  Using Duhamel's principle we may rewrite \eqref{eqnforv} as an explicit integral equation involving the heat kernel $H$ for $L$:
\begin{equation} \label{eqn:integral-eqn-for-v}
 v(t) = \underbrace{\int_0^t H(t-s) (F(v) + I) ds}_{ := \Psi( v ) }. 
\end{equation}
Note the definition of the map $\Psi$ in this equation.

For positive parameters $\mu$ and $T$ to be specified define a subspace $\sZ_{\mu,T}$ of $C^{2m,0;\alpha}(M \times [0,T])$ by
\[ \sZ_{\mu,T} = \left\{ u \in C^{2m,0;\alpha}(M \times [0,T]): u(x,0) = 0, ||u||_{C^{2m,0;\alpha}(M \times [0,T])} \leq \mu. \right\}. \]
This is a closed subset of the Banach space $C^{2m,0;\alpha}(M \times [0,T])$.

We now check the basic mapping property of $\Psi$.  Suppose that $u \in \sZ_{\mu,T}$, and set $v = \Psi u$.  Now $v$ satisfies
\begin{equation*}
\left\{ \begin{aligned}
\partial_t v + L v &= F(u) + I, \\
		v(0) &= 0. 
\end{aligned} \right.
\end{equation*}
By the mapping properties of $F$, $F(u)+I \in C^{0,0;\alpha}(M\times[0,T])$.  Consequently the Schauder estimate implies $v \in C^{2m,1;\alpha}(M \times [0,T]) \subset C^{2m,0;\alpha}(M \times [0,T])$.  This proves
\[ \Psi: \sZ_{\mu,T} \longrightarrow C^{2m,0;\alpha}(M \times [0,T]). \]

We now prove that $\Psi$ is in fact an automorphism of $\sZ_{\mu,T}$ for $\mu$ and $T$ small enough. 

\begin{lemma} $\Psi$ is an automorphism for $\mu$ and $T$ sufficiently small.
\end{lemma}
\begin{proof}
Let $u \in \sZ_{\mu,T}$.  Define
\begin{align*}
 v_1 &= \int_0^t H(t-s) F(u) ds, \; \mbox{and}\\
 v_2 &= \int_0^t H(t-s) I ds.
\end{align*}
It suffices to estimate $v_1$ and $v_2$.  Now $v_1$ is a solution to
\begin{equation*}
\left\{ \begin{aligned}
\partial_t v_1 + L v_1 &= F(u), \\
		v_1(0) &= 0. 
\end{aligned} \right.
\end{equation*}
As a consequence, the Schauder estimate, followed by the estimate for $F$ implies
\begin{align*}
||v_1||_{C^{2m,0;\alpha}(M \times [0,T])} & \leq ||v_1||_{C^{2m,1;\alpha}(M \times [0,T])} \\
& \leq K ||v_1||_{C^{0,0;\alpha}(M \times [0,T])}\\ 
& \leq K C ||u||^2_{C^{2m,0;\alpha}(M \times [0,T])}\\
& \leq K C \mu ||u||_{C^{2m,0;\alpha}(M \times [0,T])}.
\end{align*}
We may now choose $\mu$ so small that $K C \mu < \frac{1}{2}$.  This proves
\[ ||v_1||_{C^{2m,0;\alpha}(M \times [0,T])} \leq \frac{\mu}{2}. \]
This estimate persists upon taking $T$ smaller as $K$ and $C$ both remain bounded as $T \rightarrow 0^+$.

Regarding the estimate for $v_2$, observe that we have a solution to
\begin{equation*}
\left\{ \begin{aligned}
\partial_t v_2 + L v_2 &= I, \\
		v_2(0) &= 0. 
\end{aligned} \right.
\end{equation*}
where $I$ is smooth and independent of $t$.  By parabolic regularity, $v_2$ is smooth on $M \times [0,T]$, and the norm of any number of derivatives of $v_2$ may be bounded in terms of the norms of $I$.  Fixing any $x$, we may write the equation as
\[ v_2(x,t) - v_2(x,0) = \int_0^t \partial_t v_2(x, s) ds \int_0^t I(x)-Lv_2(x,s) ds. \]
We may now estimate the $C^{2m,0;\alpha}(M \times [0,T])$ norm of $v_2$.  For example,
\begin{align*}
||v_2||_{L^{\infty}(M \times [0,T])} &\leq \int_0^T ||I||+||v_2||_{C^{2m,0;\alpha}(M \times [0,T])} ds \\
&= C_1 ||I|| T, 
\end{align*}
and similarly for the other spatial derivatives.  Thus it remains to estimate the parabolic H\"older seminorm $[v_2]_{\alpha;M}$.  The time derivative of $v_2$ is bounded in terms of the norm of $||I||$.  Therefore the following interpolation inequality for parabolic H\"older spaces:
\begin{align*} [v_2]_{\alpha;M} & \leq (2 ||v_2||_{L^{\infty}(M \times [0,T])})^{1-\frac{\alpha}{2}} ||\partial_t v_2||^{\frac{\alpha}{2}}_{L^{\infty}(M \times [0,T])} \\
& \hspace{1 in} + (2 ||v_2||_{L^{\infty}(M \times [0,T])})^{1-\alpha} || \nabla v_2||_{L^{\infty}(M \times [0,T])}^{\alpha},
\end{align*}
allows us to conclude that $||v_2||_{C^{2m,0;\alpha}(M \times [0,T])} \leq C_3 T$ for some constant $C_3$ independent of $T$.  We now take $T$ small enough so that
\[ ||v_2||_{C^{2m,0;\alpha}(M \times [0,T])} \leq \frac{\mu}{2}. \]

Thus $\Psi: \sZ_{\mu,T} \longrightarrow \sZ_{\mu,T}$ for $t \in [0, T]$, as required.
\end{proof}

We now prove that $\Psi$ is a contraction on $\sZ_{\mu,T}$. 
\begin{lemma}
$ \Psi: \sZ_{\mu,T} \longrightarrow \sZ_{\mu,T} $ is a contraction.
\end{lemma}
\begin{proof}
We estimate for $u, v \in \sZ_{\mu,T}$ exactly as before, observing that the constants that occur are the same as in the previous proof.  In particular, applying the Schauder estimate followed by the estimate for $F$, we find
\begin{align*}
 || \Psi u &- \Psi v ||_{C^{2m,0;\alpha}(M \times [0,T])} \\
&\leq || \Psi u - \Psi v ||_{C^{2m,1;\alpha}(M \times [0,T])}\\
&= \left|\left| \int_0^t H(t-s)[F(u)(s)-F(v)(s)] ds \right|\right|_{C^{2m,1;\alpha}(M \times [0,T])}\\
& \leq K C \max \left\{ ||u||_{C^{2m,0;\alpha}(M \times [0,T])}, ||v||_{C^{2m,0;\alpha}(M \times [0,T])} \right\} || u - v||_{C^{2m,0;\alpha}(M \times [0,T])}\\
& \leq KC \mu || u - v ||_{C^{2m,0;\alpha}(M \times [0,T])}\mu 
\end{align*}
In the previous lemma $\mu$ was chosen so small that $K C \mu < \frac{1}{2}$, verifying the lemma.
\end{proof}
We now come to the short-time existence result 

\begin{thm}
For $T$ and $\mu$ sufficiently small, there is a unique solution $v \in C^{2m,1;\alpha}(M \times [0,T])$ to \eqref{eqnforv} in $\sZ_{\mu,T}$.  Consequently, there is a short-time solution to \eqref{adjustedflow}.
\end{thm}
\begin{proof}
The previous lemmas allow us to apply the Banach fixed point theorem in the ball $\sZ_{\mu,T}$, to obtain a fixed point $v$ of \eqref{eqn:integral-eqn-for-v}.  The Schauder estimate applied to this equation shows that in fact $v \in C^{2m,1;\alpha}(M \times [0,T])$.  Taking $g = h + v$, we now have a short-time solution to the adjusted flow \eqref{adjustedflow} in $C^{2m,1;\alpha}(M \times [0,T])$. 
\end{proof}

We remark that the solution $v$ is unique within the ball $\sZ_{\mu,T}$.  It is unclear to what extent the solution $g$ is unique when $2m > 2$.


\subsection{Regularity}

In the previous section we found a solution $g$ to \eqref{adjustedflow} in \\ $C^{2m,1;\alpha}(M \times [0,T])$.  Parabolic regularity implies that $g$ is in fact smooth.

\begin{thm}
Let $g\in C^{2m,1;\alpha}(M \times [0,T])$ be a solution to the adjusted flow. Then $g\in C^{\infty,\infty}(M \times [0,T])$.
\end{thm}
\begin{proof}
The proof is based on a bootstrap argument.  Return to the strongly parabolic equation \eqref{adjustedflow}. We write this abstractly as an equation in the form
\[ \partial_t g + \sum_{|\beta| = 0}^{2m} a_{\beta}( h, g) D^{\beta} g = 0. \]
where the coefficients $a_{\beta}$ are contractions of $h$, $g$, their inverses, and covariant derivatives of order up to $2m-1$. As $g \in C^{2m,1;\alpha}(M \times [0,T])$, we find that at worst $D^{\gamma} a_{\beta} \in C^{0,0;\alpha}(M \times [0,T])$, where $|\gamma| = 1$.  Applying Proposition \ref{thm:parabolicregularity}, we obtain that $D^{\gamma} g \in C^{2m,1;\alpha}(M \times [0,T])$ for all $|\gamma| = 1$, improving the spatial regularity.  Bootstrapping, and then using the equation to improve regularity in $t$ gives the desired result.
\end{proof}


\subsection{The Generalized DeTurck Trick}

It remains to show that a smooth solution to \eqref{adjustedflow} gives rise to a smooth solution to \eqref{originalflow}.  We note that \eqref{adjustedflow} is in the form required by Lemma \ref{deturck} so by pulling the metric back by the flow generated by $-W$ we find a solution to \eqref{originalflow}.


\section{Applications} \label{applications}

We now apply Theorem \ref{maintheorem} to flow by powers of the Laplacian of the Ricci tensor, and to flow using the ambient obstruction tensor.  As a special case of the latter flow, we have short-time existence for a type of Bach flow.

\subsection{Flow by Powers of the Laplacian of Ricci}

Recall Theorem \ref{plaprictheorem}:

\medskip
\noindent \textbf{Theorem \ref{plaprictheorem}.}
\emph{
Let $h$ be a smooth metric on a compact manifold M and let $p \geq 0$ be an integer.  Then there is a smooth short-time solution to the following flow:
\begin{equation}
\tag{$\ast \ast$}
\left\{
\begin{aligned}
\dvar{t} g &= 2 (-1)^{p+1} \Delta^p Ric\\
g(0) &= h.
\end{aligned}
\right.
\end{equation}
}

\medskip

Since the right hand side is a polynomial natural tensor, to prove Theorem \ref{plaprictheorem} it remains to find an appropriate vector field for use with the DeTurck trick and to show that the resulting flow has a strongly parabolic linearization.  Since covariant derivatives commute with Lie derivatives modulo curvature, we can generalize the vector field $V$ used for the original Ricci flow:
\begin{equation} \label{diffvf}
V^k = g^{pq} (\Gamma_{pq}^k - \tilde{\Gamma}_{pq}^k).
\end{equation}
Let
\begin{equation*}
W = (-1)^p \Delta^p V
\end{equation*}
and consider the flow:
\begin{equation} \label{plapricdeturck}
\left\{
\begin{aligned}
\dvar{t} g &= 2 (-1)^{p+1} \Delta^p Ric + \Lie{W}{g} \\
g(0) &= h.
\end{aligned}
\right.
\end{equation}

We now need to show that the linearization of the right hand side is strongly elliptic.  First, we calculate the linearization.  In what follows, Let 
\[ L^p = g^{r_1 s_1} \cdots g^{r_p s_p} \partial_{r_1} \partial_{s_1} \cdots \partial_{r_p} \partial_{s_p},\]
with the usual convention that a tilde implies the use of $h$ instead of $g$.

\begin{lemma} Let $g = h+ sv$ where $v$ is a smooth symmetric two-tensor and $s$ is a real parameter.  Working locally and focusing only on the top order terms, we have
\begin{equation*} 
\lin 2 (-1)^{p+1} \Delta^p Ric_{jk} + \left(\Lie{W}{g}\right)_{jk}
	= (-1)^p \tilde{L}^{p+1} v_{jk} + \sP(\partial^{2p+1} v, \partial^{2p+2} h, h^{-1}).
\end{equation*}
\end{lemma}

\begin{proof}
The main step is showing that the adjustment by the Lie derivative eliminates various problematic terms.  Working in coordinates we compute
\begin{equation*}
(\Delta^p Ric)_{jk} = -\frac{1}{2} L^{p+1} g_{jk}
	+ \frac{1}{2} L^p \left[(g^{il}(\partial_{ik}^2 g_{lj} + \partial_{jl}^2g_{ik} - \partial_{jk}^2 g_{il})\right]
	+ \sP(\partial^{2p+1} g, g^{-1}),
\end{equation*}
and
\begin{equation*}
\begin{split}
\left(\Lie{W}{g}\right)_{jk} &= (-1)^p \Delta^p \left(\Lie{V}{g} \right)_{jk}
						+ \sP(\partial^{2p+1} g, g^{-1}, \partial^{2p} h, h^{-1}) \\
			&= (-1)^p L^p g^{il}(\partial_{kl}^2 g_{ij} + \partial_{jl}^2 g_{ik} - \partial_{jk}^2 g_{il})
					+ \sP(\partial^{2p+1} g, g^{-1}, \partial^{2p+2}h, h^{-1}),
\end{split}
\end{equation*}
so that
\begin{equation*}
2 (-1)^{p+1} \Delta^p Ric_{jk} + \left(\Lie{W}{g}\right)_{jk}
		= (-1)^p L^{p+1} g_{jk} + \sP(\partial^{2p+1} g, g^{-1}, \partial^{2p+2}h, h^{-1}).
\end{equation*}
Calculating the linearization of this gives us the result.
\end{proof}

Now that we have the linearization, we can verify that it is strongly elliptic by showing that it satisfies \eqref{strongel}.  Making the necessary substitutions to get the principal symbol for our operator and noting that at top order, our system is uncoupled, we have
\begin{equation*}
\delta^{KL}h^{j_1k_1} h^{j_2 k_2} \cdots h^{j_{p+1} k_{p+1}}
		\xi_{j_1} \xi_{k_1} \xi_{j_2 k_2} \cdots \xi_{j_{p+1} k_{p+1}} \eta_{K} \eta_L
				= |\xi|^{2(p+1)}_h |\eta|^2,
\end{equation*}
where $1 \leq K, L \leq n(n+1)/2$ index the components of the metric.  This shows that our operator is strongly elliptic.  Therefore, Theorem \ref{maintheorem} applies and we may conclude that \eqref{plapricflow} admits a smooth solution for a short time.

As mentioned, if $p = 0$, we have Ricci flow.  When $p=1$, the steady state solutions are ``Ricci-harmonic'' metrics.

\subsection{Flow by the Ambient Obstruction Tensor}

Recall Theorem \ref{ambienttheorem}:

\medskip
\noindent \textbf{Theorem \ref{ambienttheorem}.}
\emph{
Let $h$ be a smooth metric on a compact manifold of even dimension $n \geq 4$ and let
\[
c_n = \frac{1}{2^{\frac{n}{2}-1}\left(\frac{n}{2}-2\right)!(n-2)(n-1)}.
\]
Then there is a smooth short-time solution to the following flow:
\begin{equation}
\tag{$\ast \! \ast \! \ast$}
\left\{
\begin{aligned}
\dvar{t} g &= \amb_n + c_n (-1)^{\frac{n}{2}} (\Delta^{\frac{n}{2}-1} S) g  \\
g(0) &= h.
\end{aligned}
\right.
\end{equation}
}

\medskip

\begin{proof}
Using \eqref{obstructiondef} to write the flow equation out a bit, we have
\begin{equation} \label{ambflowexpanded}
\dvar{t} g = c_n (n-1) 2 (-1)^{\frac{n}{2}} \Delta^{\frac{n}{2}-1} Ric
			+ c_n (n-2) (-1)^{\frac{n}{2}-1} \Delta^{\frac{n}{2}-2} \nabla^2 S + T_{n-1}.
\end{equation}
As in the previous section, the right hand side is a polynomial natural tensor and all we need to do is find a vector field to use in the DeTurck trick and show that the resulting flow has a strongly parabolic linearization.  We deal with the first term on the right side in \eqref{ambflowexpanded} as we did for powers of the Laplacian of Ricci above.  The second term can also be expressed as the Lie derivative of a vector field, and so will also be absorbed.  In particular, note that
\begin{align*}
\Delta^{\frac{n}{2}-2} \nabla^2 S &= \nabla^2 \Delta^{\frac{n}{2}-2} S + T_{n-2} \\
				&= \frac{1}{2} \Lie{\nabla(\Delta^{\frac{n}{2}-2} S)^{\sharp}}{g} + T_{n-2},
\end{align*}
where $T_{n-2}$ is a polynomial natural tensor, and so with $V$ defined in \eqref{diffvf} we let
\begin{equation*}
W = c_n (n-1) (-1)^{\frac{n}{2}-1} \Delta^{\frac{n}{2}-1} V
			+ \frac{c_n (n-2)(-1)^{\frac{n}{2}}}{2}\Delta^{\frac{n}{2}-2} S.
\end{equation*}
The adjusted flow
\begin{equation*}
\left\{
\begin{aligned}
\dvar{t} g &= \amb_n + c_n (-1)^{\frac{n}{2}} \Delta^{\frac{n}{2}-1} S g + \Lie{W}{g} \\
g(0) &= h
\end{aligned}
\right.
\end{equation*}
differs from \eqref{plapricdeturck} at top order only by a positive constant and so can be solved by the same means.
\end{proof}

When $n = 4$, \eqref{ambflow} becomes the following flow involving the Bach tensor $B$:
\begin{equation*}
\left\{
\begin{aligned}
\dvar{t} g &= B + \frac{1}{12} \Delta S g \\
g(0) &= h.
\end{aligned}
\right.
\end{equation*}
While this flow, and the more general flow by the ambient obstruction tensor, may not be the first choice for ``Bach Flow,'' it does have steady state solutions which are Bach flat, or more generally cause the ambient obstruction tensor to vanish.  Further, the vanishing power of the Laplacian of scalar curvature implies these metrics constant scalar curvature.  These results follow from the fact that the ambient obstruction tensor is trace free.  Thus, tracing the flow at the steady state shows that the given power of the Laplacian of $S$ must vanish, which implies that $S$ must be constant upon integration by parts.  Once we know this, it follows that the ambient obstruction tensor is also zero.  In particular, steady state solutions to this Bach flow are Bach-flat constant scalar curvature metrics.

Finally, we explore how the conformal invariance affects flow by the ambient obstruction tensor.  We find that, while an alternative geometric flow can be constructed, it does not seem to yield much further insight.

Let $\hat{g} = \rho^2 g$, where $\rho$ is a positive function to be determined on $M \times [0,T]$.  If $g$ solves \eqref{ambflow} then
\begin{align*}
\dvar{t} \hat g &= \rho^2 \dvar{t} g + 2 \rho \dvar{t} \rho g \\
			&= \rho^2 \left(\amb_n + c_n (-1)^{\frac{n}{2}} (\Delta^{\frac{n}{2}-1} S) g\right)
				+ 2 \rho \dvar{t} \rho g \\
			&= \rho^n \hat{\amb}_n
				+ \left(c_n (-1)^{\frac{n}{2}} (\Delta^{\frac{n}{2}-1} S) \rho^2
										+ 2 \rho \dvar{t} \rho \right) g.
\end{align*}
Requiring that the second term vanish, we may solve a simple ODE for $\rho$ at each point of $M$.  If we choose our initial condition for $\rho$ to be $\rho|_{t=0} = 1$, we find
\[
\rho = e^{-\frac{1}{2}\int_0^t \phi d \tau}
\]
where $\phi = c_n(-1)^{\frac{n}{2}} (\Delta^{\frac{n}{2}-1} S)$.
This expression still involves $g$.  To get everything in terms of $\hat{g}$, we would need to use the conformal change formula for $\Delta^{\frac{n}{2}-1} S$.  On one hand, this gives us a solution to a flow of the form
\begin{equation*}
\left\{
\begin{aligned}
\dvar{t} g &=  \kappa \amb_n \\
g(0) &= h,
\end{aligned}
\right.
\end{equation*}
but on the other hand, $\kappa$ is a complicated expression involving top order derivatives of the metric and seems somewhat further removed from an ideal form of ambient obstruction flow than \eqref{ambflow}.


 \bibliographystyle{amsalpha}
 \bibliography{geomflow}

\end{document}